\documentclass[11pt]{amsart}

\usepackage[linktocpage=true, colorlinks=true, linkcolor=blue, citecolor=red, urlcolor=green]{hyperref}

\usepackage[margin=3cm]{geometry}

\usepackage{amssymb,mathrsfs,stmaryrd,tikz}
\usetikzlibrary{arrows,positioning}

\newtheorem{theorem}{Theorem}[section]
\newtheorem{lemma}[theorem]{Lemma}
\newtheorem{proposition}[theorem]{Proposition}
\newtheorem{corollary}[theorem]{Corollary}

\newtheorem{claim}[theorem]{Claim}

\theoremstyle{definition}
\newtheorem{definition}[theorem]{Definition}

\newtheorem{remark}[theorem]{Remark}

\newcommand\ph\varphi
\newcommand\up{{\uparrow}}
\newcommand\dn{{\downarrow}}
\newcommand\sd\triangle
\renewcommand\le\leqslant
\renewcommand\ge\geqslant
\newcommand\imp\leqqslant
\newcommand\bimp\leftrightarrow
\newcommand\emp\varnothing
\newcommand\then\Rightarrow
\newcommand\sV{{\mathscr V}}

\newcommand\IPC{\ensuremath{\mathsf{IPC}}}
\newcommand\HB{\ensuremath{\mathsf{HB}}}
\newcommand\Sh{\ensuremath{\mathsf S_{\HB}}}
\newcommand\She{\ensuremath{\mathsf S}}
\newcommand\F{\ensuremath{\mathsf F_{\HB}}}
\newcommand\Fi{\ensuremath{\mathsf F}}
\newcommand\Log{\ensuremath{\mathsf L}}
\newcommand\CPC{\ensuremath{\mathsf{CPC}}}
\newcommand\sfour{\ensuremath{\mathsf{S4}}}
\newcommand\Grz{\ensuremath{\mathsf{Grz}}}
\newcommand\bb{\ensuremath{\mathsf{bb}}}

\DeclareMathOperator\Int{{\mathsf{int}}}
\DeclareMathOperator\Cl{{\mathsf{cl}}}
\DeclareMathOperator\closure{{\mathsf{cl}}}
\DeclareMathOperator\Op{{\mathsf{Op}}}
\DeclareMathOperator\Up{{\mathsf{Up}}}

\newcommand\coboundary\eth

\newcommand\sF{{\mathscr F}}

\newcommand\bg{\beta}
\newcommand\cg{\gamma}
\newcommand\ag{\alpha}
\newcommand\dg{\delta}

\newcommand\kg{\varkappa}

\newcommand\fH{{\mathbb H}}
\newcommand\fA{{\mathbb A}}
\newcommand\fB{{\mathbb B}}

\newcommand\cV[1]{\left\llbracket#1\right\rrbracket}

\begin{document}

\title[A negative solution of Kuznetsov's problem for bi-Heyting algebras]{A negative solution of Kuznetsov's problem for varieties of bi-Heyting algebras}

\author{G. Bezhanishvili}
\address{New Mexico State University}
\email{guram@nmsu.edu}

\author{D. Gabelaia}
\address{TSU Razmadze Mathematical Institute}
\email{gabelaia@gmail.com}

\author{M. Jibladze}
\address{TSU Razmadze Mathematical Institute}
\email{mamuka.jibladze@gmail.com}


\begin{abstract}
We show that there exist (continuum many) varieties of bi-Heyting algebras that are not generated by their complete members. It follows that there exist (continuum many) extensions of the Heyting-Brouwer logic $\HB$ that are topologically incomplete. This result provides further insight into the long-standing open problem of Kuznetsov by yielding a negative solution of the reformulation of the problem from extensions of $\IPC$ to extensions of $\HB$.
\end{abstract}

\maketitle


\section{Introduction}

Intermediate logics are the logics that are situated between the intuitionistic propositional calculus $\IPC$ and the classical propositional calculus $\CPC$. This is a large class of logics that has been thoroughly investigated over the years. One problem, however, has eluded solution for many years: it is a long-standing open problem of Kuznetsov \cite{Kuz74} whether there exists an intermediate logic that is topologically incomplete.

There are a number of results related to Kuznetsov's problem. It is well known (see, e.g., \cite{RS63} or \cite{CZ97}) that the G\"odel translation embeds $\IPC$ into Lewis's modal system $\sfour$, yielding a close correspondence between intermediate logics and normal extensions of $\sfour$. In \cite{Fin74a} Fine constructed a Kripke incomplete normal extension of $\sfour$. In \cite{Ger75a} Gerson showed that this system is also topologically incomplete, and in \cite{Ger75b} he constructed a Kripke incomplete normal extension of $\sfour$ that is topologically complete. In \cite{She77} Shehtman utilized Fine's results to construct a Kripke incomplete intermediate logic, and in \cite{She80} (see also \cite{She98}) he constructed a topologically incomplete normal extension of Grzegorczyk's modal system $\Grz$ (which is the largest modal companion of $\IPC$). These methods do not appear to be sufficient for resolving Kuznetsov's problem.

In this note we concentrate on Kuznetsov's problem for a class of logics that is obtained from intermediate logics by enriching the language with an additional binary connective of coimplication. It is well known that algebraic models of $\IPC$ are Heyting algebras, which are bounded distributive lattices with an additional binary operation of implication. It was pointed out already by McKinsey and Tarski \cite{MT46} that unlike Boolean algebras, Heyting algebras are not symmetric in that the operation of implication does not have a counterpart. This non-symmetry can be overcome by adding the operation of coimplication to the signature of Heyting algebras. The resulting algebras are known as bi-Heyting algebras. The corresponding logical system was constructed by Rauszer \cite{Rau73} under the name of Heyting-Brouwer logic (HB-logic for short).

We will consider Kuznetsov's problem for the logics intermediate between the Heyting-Brouwer logic $\HB$ and the classical propositional calculus $\CPC$. Fine's intermediate logic $\Fi$ and Shehtman's intermediate logic $\She$ have obvious HB-counterparts, which we denote by $\F$ and $\Sh$, respectively. Let $\mathcal V({\F})$ and $\mathcal V(\Sh)$ be the corresponding varieties of bi-Heyting algebras. Our main result shows that no variety in the interval $[\mathcal V({\F}),\mathcal V({\Sh})]$ is generated by its complete members. It follows from a result of Litak \cite{Lit02} that there are continuum many such varieties. Thus, there are continuuum many varieties of bi-Heyting algebras that are not generated by their complete members. This implies that there are continuum many extensions of $\HB$ that are topologically incomplete. In particular, $\Sh$ is a finitely axiomatizable extension of $\HB$ that is topologically incomplete. Thus, as with the extensions of $\sfour$, Kuznetsov's problem has a negative solution for the extensions of $\HB$.

The paper is organized as follows. In \S 2 we recall all the necessary background, including Esakia representation of Heyting algebras and bi-Heyting algebras. This will be our main technical tool. In \S 3 we recall the Fine frame, construct the corresponding Esakia space, and show that the algebra of its clopen upsets is a bi-Heyting algebra. We also introduce the HB-counterpart $\F$ of the Fine logic $\Fi$. In \S 4 we consider the HB-counterpart $\Sh$ of the Shehtman logic $\She$ and show that $\Sh$ is properly contained in $\F$. In \S 5 we give semantic characterizations of the Gabbay-de Jongh axiom $\bb_2$ which plays an important role in our considerations. In \S 6 we derive in $\IPC$ some consequences of Shehtman's axioms. Finally, in \S 7 we prove our main result that no variety of bi-Heyting algebras in the interval $[\mathcal V(\F),\mathcal V(\Sh)]$ is generated by its complete members. As a consequence, we obtain that all extensions of $\HB$ in the interval $[\Sh,\F]$ are topologically incomplete. It follows from a result of Litak \cite{Lit02} that there are continuum many such extensions of $\HB$. We also show that there are continuum many intermediate logics that are incomplete with respect to the class of their complete bi-Heyting algebras.\footnote{We were notified by Guillaume Massas that he has obtained a similar result using different technique.} This generalizes the results of Shehtman \cite{She77} and Litak \cite{Lit02} about the existence of Kripke incomplete intermediate logics.

\section{Preliminaries}


\subsection{\texorpdfstring{$\IPC$}{IPC}, Heyting algebras, and Esakia spaces}


We recall that a \emph{Heyting algebra} is a bounded distributive lattice $\fH$ with an additional binary operation $\to$ satisfying for all $a,b,x\in \fH$:
\[
a\land x\le b \mbox{ iff } x\le a\to b.
\]
It is well known that Heyting algebras are the algebraic models of $\IPC$. A class of algebras is a \emph{variety} if it is closed under subalgebras, homomorphic images, and products. Each intermediate logic $\Log$ gives rise to the variety $\mathcal V(\Log)$ consisting of the Heyting algebras that are algebraic models of $\Log$. By the standard Lindenbaum-Tarski construction, each interemdiate logic is complete with respect to $\mathcal V(\Log)$.

Typical examples of Heyting algebras come from topological spaces. For a topological space $X$, the set $\Op(X)$ of open subsets of $X$ forms a Heyting algebra. The lattice operations on $\Op(X)$ are given by set-theoretic intersection and union, and the implication is given by
\[
U\to V = \Int((X\setminus U)\cup V) = X\setminus\Cl(U\setminus V),
\]
where $\Int$ and $\Cl$ denote topological interior and closure.

\begin{theorem} [Stone Representation]
Each Heyting algebra $\fH$ is isomorphic to a subalgebra of $\Op(X)$ for some topological space $X$.
\end{theorem}

A \emph{Kripke frame} for $\IPC$ is a partially ordered set $(X,\leq)$. As usual, for $S\subseteq X$, let
\[
\up S=\{x\in X\mid\exists s\in S : s\le x\} \mbox{ and } \dn S=\{x\in X\mid\exists s\in S : x\le s\}.
\]
If $S=\{s\}$, we simply write $\up s$ and $\dn s$. We say that $S$ is an \emph{upset} if $S=\up S$
and $S$ is a \emph{downset} if $S=\dn S$.
Then the upsets form a topology on $X$ known as the \emph{Alexandroff topology}, where each point has a least open neighborhood. In this topology the closure of $S$ is $\dn S$.

Let $\Up(X)$ be the set of upsets of $X$. Then $\Up(X)=\Op(X)$, so $\Up(X)$ is a Heyting algebra, where
\[
U\to V = X\setminus\dn(U\setminus V) = \{x\mid \up x\cap U\subseteq V\}.
\]

\begin{theorem} [Kripke Representation]
Each Heyting algebra $\fH$ is isomorphic to a subalgebra of $\Up(X)$ for some Kripke frame $X$.
\end{theorem}

To recover the image of $\fH$ in $\Up(X)$, we need to introduce additional structure on $X$.

\begin{definition}
An \emph{Esakia space} is a triple $X=(X,\tau,\le)$ where $(X,\tau)$ is a Stone space (compact, Hausdorff, zero-dimensional space), $(X,\le)$ is a Kripke frame, and the order $\le$ is \emph{continuous}, meaning that:
\begin{itemize}
\item[(i)] $\up x$ is closed for each $x\in X$;
\item[(ii)] $U$ clopen implies $\dn U$ is clopen.
\end{itemize}
\end{definition}

The clopen upsets of an Esakia space $X$ form a subalgebra of $\Up(X)$, and we have the following representation of Heyting algebras:

\begin{theorem} [Esakia Representation]
Each Heyting algebra is isomorphic to the algebra of clopen upsets of some Esakia space.
\end{theorem}

While we will not need it here, this representation can be extended to a full duality (see \cite{Esa74,Esa19}). As a consequence, we obtain that each intermediate logic is sound and complete with respect to Esakia spaces.


\subsection{\texorpdfstring{$\HB$}{HB}, bi-Heyting algebras, and BH-spaces}

As we pointed out in the introduction, bi-Heyting algebras are obtained by enriching the signature of Heyting algebras by a binary operation
$\leftarrow$ of coimplication satisfying
\[
a\lor x\ge b \mbox{ iff } x\ge a\leftarrow b.
\]
Bi-Heyting algebras are algebraic models of the Heyting-Brouwer system $\HB$ introduced by Rauszer \cite{Rau73}. The axiomatization we give below is taken from Wolter \cite{Wol98}.
We use the standard abbreviations $\neg p := p\to\bot$ and ${\sim}p := \top\leftarrow p$.

\begin{definition}
\begin{enumerate}
\item[]
\item The \emph{Heyting-Brouwer calculus} $\HB$ is the least set of formulas (in the propositional language enriched with $\leftarrow$) containing $\IPC$, the axioms
\begin{itemize}
\item $p\to(q\vee(q\leftarrow p))$
\item $(q\leftarrow p)\to{\sim}(p\to q)$
\item $(r\leftarrow(q\leftarrow p))\to((p\vee q)\leftarrow p)$
\item $\neg(q\leftarrow p)\to(p\to q)$
\item $\neg(p\leftarrow p)$
\end{itemize}
and closed under the rules of substitution, modus ponens, and $\dfrac{\varphi}{\neg{\sim}\varphi}$.
\item An \emph{intermediate HB-logic} is a consistent set of formulas containing $\HB$ and closed under these three rules.
\end{enumerate}
\end{definition}

As with intermediate logics, each intermediate HB-logic $\Log$ is complete with respect to the variety $\mathcal V(\Log)$ of bi-Heyting algebras.

If $(X,\le)$ is a Kripke frame, then $\Up(X)$ is a bi-Heyting algebra, where
\[
U\leftarrow V = \up(V \setminus U) = \{x \mid \dn x\cap V\not\subseteq U\}.
\]

We next introduce Esakia spaces that are dual to bi-Heyting algebras. We call them BH-spaces (for bi-Heyting or Brouwer-Heyting).

\begin{definition}
A \emph{BH-space} is an Esakia space $X$ such that $U$ clopen implies $\up U$ is clopen.
\end{definition}

The following representation of bi-Heyting algebras was obtained by Esakia \cite{Esa75,Esa78}.

\begin{theorem}\label{thm:BH}
A Heyting algebra $\mathbb H$ is a bi-Heyting algebra iff its Esakia space $X$ is a BH-space. In particular, each bi-Heyting algebra can be represented as the algebra of clopen upsets of a BH-space.
\end{theorem}

Consequently, BH-spaces provide an adequate semantics for intermediate HB-logics.

\subsection{Dual characterization of complete algebras}

We finish this preliminary section by recalling the dual description of complete Heyting and bi-Heyting algebras. This description will play an essential role in \S 7. We recall that a topological space is \emph{extremally disconnected} if the closure of each open set is open (and hence clopen). This is adjusted to Esakia spaces as follows.

\begin{definition}
An Esakia space is \emph{extremally order-disconnected} if the closure of each open upset is open.
\end{definition}

Complete Heyting algebras are dually characterized by extremally order-dis\-con\-nec\-ted Esakia spaces (see, e.g., \cite[Thm.~2.4.2]{BB08} and the references therein):

\begin{theorem}\label{thm:cBH}
A Heyting algebra is complete iff its Esakia space is extremally order-dis\-con\-nec\-ted. In particular, a bi-Heyting algebra is complete iff its BH-space is extremally order-disconnected.
\end{theorem}

In the Heyting algebra of clopen upsets of an extremally order-disconnected Esakia space $X$, the join is computed as the closure of the union and the meet as the greatest upset contained in the interior of the intersection. If in addition $X$ is a BH-space, then the interior of an upset is an upset (see e.g. \cite[Lem.~3.6]{HB04}), and hence joins and meets are computed as follows (see, e.g., \cite[Thm.~3.8]{HB04}):
\begin{equation}\label{eq:bihjoinmeet}
\bigvee U_i = \Cl\left(\bigcup U_i\right) \mbox{ and } \bigwedge U_i = \Int\left(\bigcap U_i\right).
\end{equation}





\section{The Fine frame and the corresponding BH-space}



As we pointed out in the introduction, Fine \cite{Fin74a} constructed a Kripke incomplete normal extension of $\sfour$ utilizing the Kripke frame $\sF = (F,\leq)$ shown in Figure~\ref{f:fine frame}, which has since become known as the \emph{Fine frame}.

\begin{figure}[ht]
\[
\begin{tikzpicture}[scale=.7,inner sep=.2]
\node (p) at (0,0) {$\bullet$};
\node[red,right] at (p.north east) {$c_0$};
\node (q) at (-2,0) {$\bullet$};
\node[red,left] at (q.north west) {$b_0$};
\node (p1) at (0,-1) {$\bullet$};
\node[red,right] at (p1.north east) {$c_1$};
\node (q1) at (-2,-1) {$\bullet$};
\node[red,left] at (q1.north west) {$b_1$};
\node (p2) at (0,-2) {$\bullet$};
\node[red,right] at (p2.north east) {$c_2$};
\node (q2) at (-2,-2) {$\bullet$};
\node (p3) at (0,-3) {$\bullet$};
\node[red,right] at (p3.north east) {$c_3$};
\node (q3) at (-2,-3) {$\bullet$};
\node (p4) at (0,-4) {$\bullet$};
\node[red,right] at (p4.north east) {$c_4$};
\node (q4) at (-2,-4) {$\bullet$};
\node (a1) at (-12, -6) {$\bullet$};
\node[red,left] at (a1.west) {$a_0$};
\node (a2) at (-10, -6) {$\bullet$};
\node[red,left] at (a2.west) {$a_1$};
\node (a3) at (-8, -6) {$\bullet$};
\node[red,left] at (a3.west) {$a_2$};
\node (a4) at (-6, -6) {$\bullet$};
\node[red,left] at (a4.west) {$a_3$};
\node (d1) at (-1,-12) {$\bullet$};
\node[red,right] at (d1.east) {$d_0$};
\node (d2) at (-1,-11) {$\bullet$};
\node[red,right] at (d2.east) {$d_1$};
\node (d3) at (-1,-10) {$\bullet$};
\node[red,right] at (d3.east) {$d_2$};
\node (d4) at (-1,-9) {$\bullet$};
\node[red,right] at (d4.east) {$d_3$};
\node at (-1,-5.5) {$\vdots$};
\node at (-4,-6) {$\cdots$};
\node at (-1,-7.5) {$\vdots$};
\draw (p) -- (p1) -- (p2) -- (p3) -- (p4) -- (0,-4.8);
\draw (q) -- (q1) -- (q2) -- (q3) -- (q4) -- (-2,-4.8);
\draw (p) -- (q2) -- (p4) -- (-.8,-4.8);
\draw (p1) -- (q3) -- (-.2,-4.8);
\draw (q) -- (p2) -- (q4) -- (-1.2,-4.8);
\draw (q1) -- (p3) -- (-1.8,-4.8);
\draw (p1) -- (a1) -- (q1);
\draw (p2) -- (a2) -- (q2);
\draw (p3) -- (a3) -- (q3);
\draw (p4) -- (a4) -- (q4);
\draw (d1) -- (d2) -- (d3) -- (d4) -- (-1,-8.2);
\draw (a1) -- (d1);
\draw (a2) -- (d2);
\draw (a3) -- (d3);
\draw (a4) -- (d4);
\end{tikzpicture}
\]
\caption{The Fine frame.}\label{f:fine frame}
\end{figure}
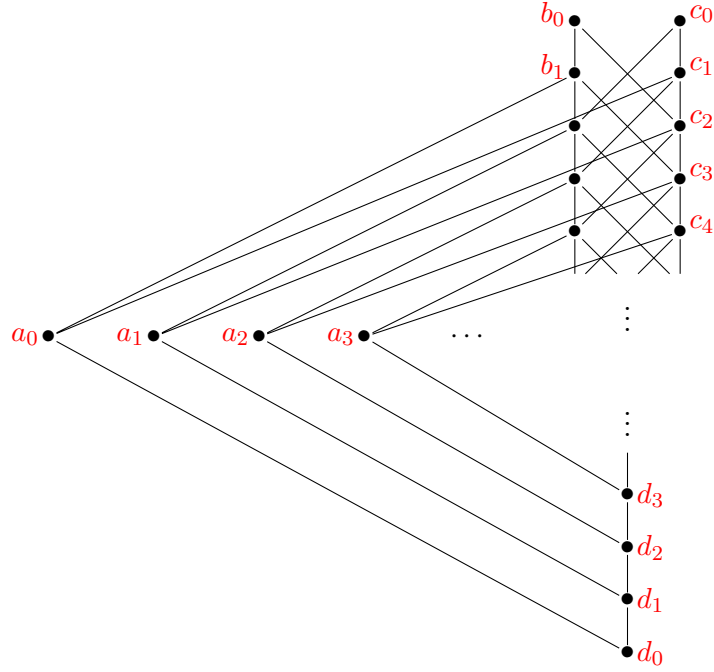
Shehtman \cite{She77} utilized the Fine frame to construct a Kripke incomplete intermediate logic. For this he introduced a particular  intuitionistic valuation on $\mathscr F$ assigning the upset $\{c_0\}$ to the propositional letter $p$ and the upset $\{b_0\}$ to $q$. These upsets generate a Heyting subalgebra inside all upsets of $\mathscr F$. In this section we explicitly describe this subalgebra using Esakia duality, and show that it is in fact a bi-Heyting algebra.

We let
\begin{eqnarray*}
A&=&\{a_i\mid i<\omega\},\\
B&=&\{b_i\mid i<\omega\},\\
C&=&\{c_i\mid i<\omega\},\\
D&=&\{d_i\mid i<\omega\}.
\end{eqnarray*}

We now define an Esakia space $X_{\mathscr F}$ based on the one-point compactification of $\mathscr F$.

\begin{definition}
Let $X_{\mathscr F} = \mathscr F\cup\{\infty\}$. The topology of $X_{\mathscr F}$ is defined in the standard way, where all elements of $\mathscr F$ are isolated and $\infty$ is the only limit point. The ordering on $X_{\mathscr F}$ is given by extending the ordering of $\mathscr F$ by putting $\infty\leq b_i, c_i$ and $d_i\leq\infty$ for all $i<\omega$, as shown in Figure~\ref{f:BH-space}.
\end{definition}

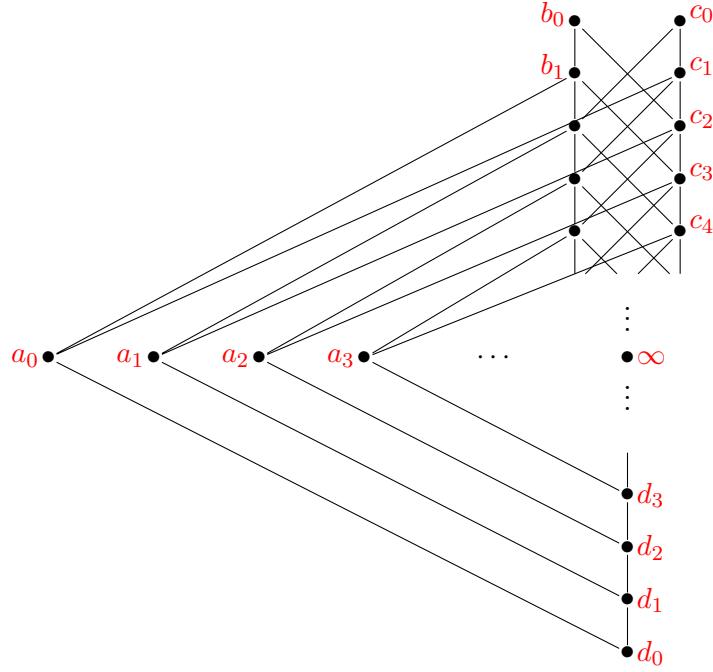
\begin{figure}[ht]
\[
\begin{tikzpicture}[scale=.7,inner sep=.2]
\node (p) at (0,0) {$\bullet$};
\node[red,right] at (p.north east) {$c_0$};
\node (q) at (-2,0) {$\bullet$};
\node[red,left] at (q.north west) {$b_0$};
\node (p1) at (0,-1) {$\bullet$};
\node[red,right] at (p1.north east) {$c_1$};
\node (q1) at (-2,-1) {$\bullet$};
\node[red,left] at (q1.north west) {$b_1$};
\node (p2) at (0,-2) {$\bullet$};
\node[red,right] at (p2.north east) {$c_2$};
\node (q2) at (-2,-2) {$\bullet$};
\node (p3) at (0,-3) {$\bullet$};
\node[red,right] at (p3.north east) {$c_3$};
\node (q3) at (-2,-3) {$\bullet$};
\node (p4) at (0,-4) {$\bullet$};
\node[red,right] at (p4.north east) {$c_4$};
\node (q4) at (-2,-4) {$\bullet$};
\node (a1) at (-12, -6.4) {$\bullet$};
\node[red,left] at (a1.west) {$a_0$};
\node (a2) at (-10, -6.4) {$\bullet$};
\node[red,left] at (a2.west) {$a_1$};
\node (a3) at (-8, -6.4) {$\bullet$};
\node[red,left] at (a3.west) {$a_2$};
\node (a4) at (-6, -6.4) {$\bullet$};
\node[red,left] at (a4.west) {$a_3$};
\node (infty) at (-1,-6.4) {$\bullet$};
\node[red,right] at (infty.east) {$\infty$};
\node (d1) at (-1,-12) {$\bullet$};
\node[red,right] at (d1.east) {$d_0$};
\node (d2) at (-1,-11) {$\bullet$};
\node[red,right] at (d2.east) {$d_1$};
\node (d3) at (-1,-10) {$\bullet$};
\node[red,right] at (d3.east) {$d_2$};
\node (d4) at (-1,-9) {$\bullet$};
\node[red,right] at (d4.east) {$d_3$};
\node at (-1,-5.5) {$\vdots$};
\node at (-3.5,-6.4) {$\cdots$};
\node at (-1,-7) {$\vdots$};
\draw (p) -- (p1) -- (p2) -- (p3) -- (p4) -- (0,-4.8);
\draw (q) -- (q1) -- (q2) -- (q3) -- (q4) -- (-2,-4.8);
\draw (p) -- (q2) -- (p4) -- (-.8,-4.8);
\draw (p1) -- (q3) -- (-.2,-4.8);
\draw (q) -- (p2) -- (q4) -- (-1.2,-4.8);
\draw (q1) -- (p3) -- (-1.8,-4.8);
\draw (p1) -- (a1) -- (q1);
\draw (p2) -- (a2) -- (q2);
\draw (p3) -- (a3) -- (q3);
\draw (p4) -- (a4) -- (q4);
\draw (d1) -- (d2) -- (d3) -- (d4) -- (-1,-8.2);
\draw (a1) -- (d1);
\draw (a2) -- (d2);
\draw (a3) -- (d3);
\draw (a4) -- (d4);
\end{tikzpicture}
\]
\caption{The one-point compactification of the Fine frame.}\label{f:BH-space}
\end{figure}

\begin{proposition}\label{prop:FisBH}
The ordered space $X_{\mathscr F}$ is a BH-space.
\end{proposition}

\begin{proof}
It is well known that $X_{\mathscr F}$ is a Stone space, with its clopen sets being finite subsets of $\mathscr F$ and cofinite subsets containing $\infty$. It is also straightforward that $\leq$ is a partial order.

To see that $X_{\mathscr F}$ is a BH-space, we first observe that $\up x$ is closed for each $x\in X_{\mathscr F}$. Note that $\up x$ is finite (and hence closed) unless $x\in D\cup\{\infty\}$, in which case $\infty\in\up x$ which also implies that $\up x$ is closed.

Let next $U\subseteq X_{\mathscr F}$ be an arbitrary clopen. If $U\cap(B\cup C)\neq\varnothing$, then $\dn U$ is a cofinite set containing $\infty$, so it is clopen. Otherwise $\infty\notin U$. Therefore, $U$ is a finite subset of $D\cup A$, so $\dn U$ is finite, and hence clopen. The case of $\up U$ is treated similarly
depending on whether or not $U\cap D$ is empty. Thus, both $\up U$ and $\dn U$ are clopen, yielding that $X_{\mathscr F}$ is a BH-space.
\end{proof}

\begin{definition}
Let $\fA$ be the algebra of clopen upsets of $X_{\mathscr F}$.
\end{definition}

As an immediate consequence of Theorem \ref{thm:BH} and Proposition \ref{prop:FisBH} we obtain:

\begin{corollary}\label{cor:biH}
$\fA$ is a bi-Heyting algebra.
\end{corollary}

\begin{remark}
The algebra $\fA$ is not complete since $X_{\mathscr F}$ fails to be extremally order-dis\-con\-nec\-ted. Indeed, $B\cup C$ is an open upset, but its closure is $B\cup C\cup\{\infty\}$ which is not open.
\end{remark}

Consider the following upsets of $\mathscr F$:

\

\begin{tabular}{llll}
$B_n:=-\dn b_n$,& $C_n:=-\dn c_n$,& $A_n:=-\dn a_n $,& $D_n:=-\dn d_n $;\\
$B_n':=\up b_n$,& $C_n':=\up c_n$,& $A_n':=\up a_n$,& $D_n':=\up d_n$.
\end{tabular}

\

The next lemma can be proved by a straightforward calculation, so we omit its proof.

\begin{lemma}\label{lem:ABCD}
In $\Up(\mathscr F)$ we have 
\begin{align*}
B_0&=C'_1=B_0'\to C_0',\\
C_0&=B'_1=C_0'\to B_0',\\
B_1&=C_2' 
 = C_0 \to (B_0\cup B_0'), \\
C_1&=B_2' 
 = B_0 \to (C_0 \cup C_0').
\intertext{Moreover, for all $n\ge0$,}
B_{n+2}&=C_{n+1}\to(B_{n+1}\cup C_n),\\
C_{n+2}&=B_{n+1}\to (C_{n+1}\cup B_n),\\
A_n&=(B_{n+2}\cap C_{n+2})\to(B_{n+1}\cup C_{n+1}),\\
D_n&=A_n\cup A_{n+1}
\intertext{and}
B'_{n+2}&=C_{n+1}\cap A_n,\\
C'_{n+2}&=B_{n+1}\cap A_n,\\
D'_n&=\bigcap_{i<n}A_i,\\
A'_n&=D'_n\cap B_{n+2}\cap C_{n+2}.
\end{align*}
\end{lemma}

\begin{theorem}
The algebra $\fA$ is isomorphic to the subalgebra $\fB$ of $\Up(\mathscr F)$ generated by $\{\{b_0\},\{c_0\}\}$.
\end{theorem}

\begin{proof}
Define $f:\fA\to\Up(\mathscr F)$ by $f(U) = U\cap\mathscr F$ for each clopen upset $U\subseteq X_{\mathscr F}$. It is obvious that $f$ is well defined. It is also not difficult to verify that $U\subseteq V$ iff $f(U)\subseteq f(V)$ for clopen $U,V$. This follows from ${\mathscr F}$ being a dense subset of $X_{\mathscr F}$.

We show that $\fB\subseteq f(\fA)$. We have that $\{b_0\}$, $\{c_0\}$ are clopen upsets of $X_{\mathscr F}$ and that $f(\{b_0\})=\{b_0\}, f(\{c_0\})=\{c_0\}$. Therefore, $\{b_0\},\{c_0\}\in f(\fA)$. Since $\fB$ is generated by $\{\{b_0\},\{c_0\}\}$, it follows that $\fB\subseteq f(\fA)$.

For the reverse inclusion, let $U$ be a clopen upset of $X_{\mathscr F}$.
If $\infty\notin U$, then $U$ contains a finite number of the $a_i, b_i, c_i$ and does not contain any of the $d_i$. Therefore, $f(U)=U$ is a finite union of the $A'_i$, $B'_i$, $C'_i$. On the other hand, if $\infty\in U$, then $f(U)=U\cap\mathscr F$ is cofinite, hence $U$ does not contain only a finite number of the $a_i$ and $d_i$ and contains all the $b_i$, $c_i$. Thus, $U$ is a finite intersection of the $A_i$ and $D_i$. But Lemma~\ref{lem:ABCD} yields that each of the above sets is an element of $\fB$, giving the reverse inclusion. Thus, $f:\fA\to\fB$ is an isomorphism.
\end{proof}

Let $\mathsf{Form}_{\IPC}$ be the set of formulas in the propositional language for $\IPC$ and let $\mathsf{Form}_{\HB}$ be the set of formulas in the propositional language for $\HB$.

\begin{definition}
\begin{enumerate}
\item[]
\item The \emph{Fine logic} is the intermediate logic
\[
\Fi=\{\varphi \in \mathsf{Form}_{\IPC} \mid \fA\vDash\varphi\}.
\]
\item Let $\mathcal V(\Fi)$ be the variety of Heyting algebras generated by $\fA$.
\item The HB-counterpart of $\Fi$ is the extension of $\HB$
\[
\F=\{\varphi \in \mathsf{Form}_{\HB} \mid \fA\vDash\varphi\}.
\]
We call $\F$ the {\em Fine logic over} $\HB$.
\item Let $\mathcal V(\F)$ be the variety of bi-Heyting algebras generated by $\fA$.
\end{enumerate}
\end{definition}

\begin{remark}
Since $\fA$ is a bi-Heyting algebra, the variety $\mathcal V(\Fi)$ is generated by bi-Heyting algebras. Consequently, $\F$ is conservative over $\Fi$, meaning that for each $\varphi\in\mathsf{Form}_{\IPC}$, we have $\Fi\vdash\varphi$ iff $\F\vdash\varphi$.
\end{remark}

\section{The Shehtman logic}\label{sec:Shehtman}

In \cite{She77} Shehtman considered certain formulas valid in the Fine frame and used them to define his Kripke incomplete logic. We already encountered these formulas in Lemma~\ref{lem:ABCD}. We follow Litak \cite{Lit02} in denoting these formulas by Greek letters.


\begin{definition}\label{def:abcd}
Starting with propositional variables $p$, $q$ we define
\[
\bg_0=q\to p,\qquad \cg_0=p\to q, 
\]
\[
\bg_1=\cg_0\to(\bg_0\lor q),\qquad \cg_1=\bg_0\to(\cg_0\lor p),
\]
and for $n\ge 0$,
\[
\bg_{n+2}=\cg_{n+1}\to(\bg_{n+1}\lor\cg_n),\qquad\cg_{n+2}=\bg_{n+1}\to(\cg_{n+1}\lor\bg_n).
\]

Furthermore, for $n\ge0$ let
\begin{eqnarray*}
\ag_n &=& (\bg_{n+2}\land\cg_{n+2})\to(\bg_{n+1}\lor\cg_{n+1}), \\
\dg_n &=& \ag_n\lor\ag_{n+1}, \\
\kg_n &=& \ag_{n+1}\to(\ag_n\lor\bg_{n+2}).
\end{eqnarray*}
\end{definition}


\begin{remark}
Compare Definition \ref{def:abcd} to Lemma \ref{lem:ABCD}.
\end{remark}


We next recall the formula $\bb_2$ of Gabbay and de Jongh \cite{GdJ74}.

\begin{definition}
Let
\begin{eqnarray*}
\bb_2 &=& [(p\to(q\lor r))\to(q\lor r)]\land \\
 & & [(q\to(p\lor r))\to(p\lor r)]\land \\
 & & [(r\to(p\lor q))\to(p\lor q)] \to (p\lor q\lor r)
\end{eqnarray*}
\end{definition}

\begin{definition}
\begin{enumerate}
\item[]
\item The \emph{Shehtman logic} is the intermediate logic
\[
\She=\IPC + [(\ag_0\to\dg_1)\to\dg_0] + \kg_0 + \bb_2.
\]
\item Let $\mathcal V(\She)$ be the variety of Heyting algebras corresponding to $\She$.
\item The HB-counterpart of $\She$ is the extension of $\HB$
\[
\Sh=\HB + [(\ag_0\to\dg_1)\to\dg_0] + \kg_0 + \bb_2.
\]
We call $\Sh$ the {\em Shehtman logic over} $\HB$.
\item Let $\mathcal V(\Sh)$ be the variety of bi-Heyting algebras corresponding to $\Sh$.
\end{enumerate}
\end{definition}

\begin{remark}
Unlike for the variety $\mathcal V(\Fi)$ corresponding to the Fine logic $\Fi$, it is unclear whether $\mathcal V(\She)$ is generated by bi-Heyting algebras. Thus, it is unclear whether $\Sh$ is conservative over~$\She$.
\end{remark}


The next theorem follows from Shehtman \cite[Lem.~3 and 5]{She77}.

\begin{theorem}\label{thm:ASnod0}
The algebra $\fA$ validates $\She$ and refutes $\delta_0$.
\end{theorem}

Since $\fA$ is a bi-Heyting algebra (see Corollary \ref{cor:biH}), we obtain:

\begin{corollary}\label{cor:ASnod0}
The algebra $\fA$ validates $\Sh$ and refutes $\delta_0$.
\end{corollary}

\begin{remark}
For the reader's convenience, we illustrate how to refute $\delta_0$ in $\fA$. Consider the valuation $\sV(p)=\{c_0\}$, $\sV(q)=\{b_0\}$ 
and note that for $n\ge 0$ we have:
\begin{align*}
\sV(\bg_n)&=-\dn\{b_n\},\\
\sV(\cg_n)&=-\dn\{c_n\},\\
\sV(\ag_n)&=-\dn\{a_n\},\\
\sV(\dg_n)&=-\dn\{d_n\}.
\end{align*}
Therefore, $\sV(\dg_0)=-\dn\{d_0\}\neq X_{\mathscr F}$.
\end{remark}

\begin{corollary}\
\begin{enumerate}
\item $\She\subseteq\Fi$ and $\delta_0\notin \She$. 
\item $\Sh\subseteq\F$ and $\delta_0\notin \Sh$. 
\end{enumerate}
\end{corollary}

\begin{remark}\label{rem:continuum}
Litak \cite{Lit02} proved that the interval $[\She,\Fi]$ contains continuum many intermediate logics. As an immediate corollary we obtain that the interval $[\Sh,\F]$ also contains continuum many HB-logics.
\end{remark}


\section{Semantic characterizations of \texorpdfstring{$\bb_2$}{bb2}}

In this section we provide algebraic and order-topological characterizations of the axiom $\bb_2$. Let $\fH$ be a Heyting algebra and $b\in\fH$. We recall that $b$ is \emph{dense} if $b\to 0=0$ and $b$ is \emph{regular} if $(b\to 0)\to 0=b$. We next relativize these two notions. Let $a,b\in\fH$ with $a\le b$. We say that \emph{$b$ is dense over $a$} if $b\to a=a$ and \emph{$b$ is regular over $a$} if $(b\to a)\to a=b$.

Using the standard shorthand notation
\[
w_a(b):=(b\to a)\to a,
\]
we have that $b$ is dense over $a$ iff $w_a(b)=1$ and $b$ is regular over $a$ iff $w_a(b)=b$. The next lemma is then immediate.


\begin{lemma}\label{l:bb2-H}
Let $\fH$ be a Heyting algebra. Then $\fH\models\bb_2$ iff for all $a,b,c\in\fH$ we have
\[
w_{a\lor b}(c)\land w_{a\lor c}(b)\land w_{b\lor c}(a)=a\lor b\lor c.
\]
\end{lemma}

Recall that for any element $u$ of a Heyting algebra $\fH$, the interval $[0,u]$ is also a Heyting algebra, where the implication is given by $x\to_uy=u\land (x\to y)$ (see, e.g., \cite[Sec.~IV.8]{RS63}).

\begin{definition}
Let $\fH$ be a Heyting algebra.
\begin{enumerate}
\item We say that property $(*)$ holds in $\fH$ if
\begin{equation*}
\text{$a\lor b=a\lor c=b\lor c=d$ and $d$ dense over $a, b, c$ imply that $d=1$.}
\end{equation*}
\item We say that $(*)$ holds hereditarily in $\fH$ if $(*)$ holds in the algebra $[0,u]$ for each $u\in\fH$.
\end{enumerate}
\end{definition}

\begin{theorem}\label{thm:bb2}
A Heyting algebra $\fH$ validates $\bb_2$ iff $(*)$ holds hereditarily in $\fH$.
\end{theorem}

\begin{proof}
Suppose $\fH$ validates $\bb_2$. To show that $(*)$ holds hereditarily in $\fH$, let $u\in\fH$,
let $a,b,c,d\le u$ be such that  $a\lor b=a\lor c=b\lor c=d$, and let $d$ be dense over $a,b,c$ in $[0,u]$. 
It is easy to see that the latter condition on $d$ is equivalent to $u\le w_a(d)\land w_b(d)\land w_c(d)$. Let $p=a\land b$, $q=a\land c$, and $r=b\land c$. We have that $a\lor b=a\lor c=b\lor c$ iff $a\le b\lor c$, $b\le a\lor c$, and $c\le a\lor b$, which implies that
\[
p\lor q=(a\land b)\lor(a\land c)=a\land(b\lor c)=a,
\]
and similarly $p\lor r=b$ and $q\lor r=c$. Therefore, $p\lor q\lor r=a\lor b\lor c=d$ and $u\le w_{p\lor q}(d)\land w_{p\lor r}(d)\land w_{q\lor r}(d)$. Now observe that
\[
(p\lor q\lor r)\to(p\lor q)=r\to(p\lor q)
\]
for any $p,q,r\in\fH$, so $w_{p\lor q}(d)=w_{p\lor q}(r)$, and similarly $w_{p\lor r}(d)=w_{p\lor r}(q)$, $w_{q\lor r}(d)=w_{q\lor r}(p)$. Therefore, $u\le w_{p\lor q}(r)\land w_{p\lor r}(q)\land w_{q\lor r}(p)$. Thus, by Lemma~\ref{l:bb2-H},
$u\le p\lor q\lor r=d$, and hence $d=1$ in $[0,u]$.

Conversely, suppose $(*)$ holds in $\fH$ hereditarily. To see that $\fH$ validates $\bb_2$, it is sufficient to show that the equality in Lemma~\ref{l:bb2-H} holds for arbitrary $p,q,r\in\fH$. Let
\[
u=w_{p\lor q}(r)\land w_{p\lor r}(q)\land w_{q\lor r}(p) \mbox{ and } d=p\lor q\lor r.
\]
We must show that $u=d$. Let $a=p\lor q$, $b=p\lor r$, and $c=q\lor r$. Then
\[
a\lor b=a\lor c=b\lor c=d.
\]
Since $x\lor y\le w_x(y)$ for any $x,y\in\fH$, we have $p\lor q\lor r \le w_{p\lor q}(r),w_{p\lor r}(q),w_{q\lor r}(p)$, so $a,b,c\le u$. As $u\le w_{p\lor q}(r)=w_{p\lor q}(p\lor q\lor r)$, we have that $d$ is dense over $a$ in $[0,u]$. Similarly, $d$ is dense over $b$ and $c$ in $[0,u]$. Because $(*)$ holds in $[0,u]$, we must have $a\lor b\lor c=1$ in $[0,u]$, so $u\le a\lor b\lor c$. But since $a,b,c\le u$, this means that $a\lor b\lor c=u$, and hence $u=d$.
\end{proof}

The dual characterization of $(*)$ is more intuitive in the language of clopen downsets.

\begin{definition}
Let $X$ be an Esakia space and $D,E$ two downsets of $X$ such that $D\subseteq E$. We say that $D$ is \emph{nowhere cofinal} in $E$ if $\dn(E\setminus D)=E$.
\end{definition}

\begin{remark}\label{rem: nowhere dense}\
\begin{enumerate}
\item It is easy to see that $D$ is nowhere cofinal in $E$ iff there is $S\subseteq E$ such that $D\cap S=\emp$ and $D\subseteq\dn S$.
\item If $D,E$ are clopen downsets, then $D$ is nowhere cofinal in $E$ iff $X\setminus D$ is dense over $X\setminus E$ in the Heyting algebra of clopen upsets of $X$.
\end{enumerate}
\end{remark}

\begin{definition}
Let $X$ be an Esakia space.
\begin{enumerate}
\item We say that property ($**$) holds in $X$ if for any clopen downsets $A,B,C,D$ of $X$,
\begin{equation*}
\text{$A\cap B=A\cap C=B\cap C=D$ and $D$ nowhere cofinal in $A, B, C$ imply that $D=\varnothing$.}
\end{equation*}
\item We say that ($**$) holds hereditarily in $X$ if ($**$) holds in every clopen downset of $X$.
\end{enumerate}
\end{definition}

\begin{lemma}\label{l:**}
Let $\fH$ be a Heyting algebra and $X$ its Esakia space.
\begin{enumerate}
\item $\fH$ satisfies $(*)$ iff $X$ satisfies $(**)$.
\item $\fH$ satisfies $(*)$ hereditarily iff $X$ satisfies $(**)$ hereditarily.
\end{enumerate}
\end{lemma}

\begin{proof}
This follows from a straightforward application of Esakia duality and Remark~\ref{rem: nowhere dense}.
\end{proof}

Putting Theorem~\ref{thm:bb2} and Lemma~\ref{l:**} together yields:

\begin{theorem}\label{thm:bb2kripke}
A Heyting algebra $\fH$ validates $\bb_2$ iff $(**)$ holds in its Esakia space $X$ hereditarily.
\end{theorem}



\section{Some syntactic calculations in \texorpdfstring{$\IPC$}{IPC}}

In this section we derive in $\IPC$ some consequences of the formulas that constitute Shehtman's axioms. In what follows, for $\alpha, \beta \in \mathsf{Form}_{\IPC}$ we write $\alpha\vdash\beta$ when $\beta$ is derivable in $\IPC$ from $\alpha$, and we write $\alpha\equiv\beta$ to mean that both $\alpha\vdash\beta$ and $\beta\vdash\alpha$ hold.

We will utilize the following substitution considered by Shehtman in \cite{She77}. Our notation is that of Litak from \cite{Lit02}:
\[
e(p)=q\lor(q\to p),\qquad e(q)=p\lor(p\to q).
\]
We then extend $e$ to all formulas by setting
\[
e(\phi\circ\psi)=e(\phi)\circ e(\psi),
\]
where $\circ=\lor,\land$ or $\to$.

The following lemma is straightforward.

\begin{lemma}\label{lem:monotbc}
$p\vdash\bg_0$, $q\vdash\cg_0$ and
$\bg_n\vdash\bg_{n+1}$, $\cg_n\vdash\cg_{n+1}$, $\bg_n\vdash\cg_{n+2}$, $\cg_n\vdash\bg_{n+2}$, $\bg_{n+1}\vdash\ag_n$, $\cg_{n+1}\vdash\ag_n$ for all $n\ge0$.
\end{lemma}

For $n\ge0$ and $\varphi\in\mathsf{Form}_{\IPC}$ define $e^n(\varphi)$ by
\[
e^0(\varphi)=\varphi\text{ and }e^{n+1}(\varphi)=e(e^n(\varphi)).
\]

\begin{lemma}\label{lem:subst}
$\varphi_n\equiv e^n(\varphi_0)$ for $\varphi=\ag$, $\bg$, $\cg$, $\dg$,
$\kg$ and $n\ge0$.
\end{lemma}

\begin{proof}
See \cite[Lem.~1]{She77} or \cite[Lem.~6]{She80}.
\end{proof}


\begin{lemma}\label{lem:ebimp}
$e^m(p\bimp q)\vdash e^{m+n}(p\bimp q)$ for all $m,n\ge0$.
\end{lemma}

\begin{proof}
We have the following chain of derivations
\[
p\bimp q\ \vdash\ p\to q\ \vdash\ p\lor(p\to q)=e(q).
\]
Similarly, $p\bimp q\vdash e(p)$. Therefore, we have the chain of derivations
\[
p\bimp q\ \vdash\ e(p)\land e(q)\ \vdash\ e(p)\bimp e(q)\ \equiv\ e(p\bimp q).
\]
Thus,
\[
e^m(p\bimp q)\vdash e^{m+1}(p\bimp q)\vdash\cdots\vdash e^{m+n}(p\bimp q).
\]
\end{proof}

\begin{lemma}\label{lem:epq}
$e(p\land q)\ \equiv\ e(p\bimp q)\land((p\to q)\lor(q\to p))\ \equiv\ p\lor q\lor(p\bimp q)$.
\end{lemma}

\begin{proof}
We have
\begin{align*}
e(p\land q)&\ \equiv\ (p\lor(p\to q))\land(q\lor(q\to p))\\
&\ \equiv\ (p\land q)\lor(p\land(q\to p))\lor(q\land(p\to q))\lor(p\bimp q)\\&\ \equiv\ p\lor q\lor(p\bimp q)
\end{align*}
and
\[
e(p\bimp q)\land((p\to q)\lor(q\to p))\ \equiv\ (e(p\bimp q)\land(p\to q))\lor(e(p\bimp q)\land(q\to p)).
\]
Since
\begin{align*}
e(p\bimp q)&\ \equiv e(p)\bimp e(q)\\
&\ \equiv\ (p\lor(p\to q))\bimp(q\lor(q\to p))\\&\ \equiv\ ((p\to q)\to(q\lor(q\to p)))\land((q\to p)\to(p\lor(p\to q))),
\end{align*}
we obtain
\begin{align*}
e(p\bimp q)\land(p\to q)
&\ \equiv\ ((p\to q)\to(q\lor(q\to p)))\land((q\to p)\to(p\lor(p\to q)))\land(p\to q)\\
&\ \equiv\ ((p\to q)\to(q\lor(q\to p)))\land(p\to q)\\
&\ \equiv\ (q\lor(q\to p))\land(p\to q)\\
&\ \equiv\ (q\land(p\to q))\lor((q\to p)\land(p\to q))\\
&\ \equiv\ q\lor(p\bimp q).
\end{align*}
Similarly,
\[
e(p\bimp q)\land(q\to p)\ \equiv\ p\lor(p\bimp q).
\]
Thus,
\[
e(p\bimp q)\land((p\to q)\lor(q\to p))\ \equiv\ p\lor q\lor(p\bimp q).
\]
\end{proof}

\begin{lemma}\label{cor:apq}
$\ag_n\ \equiv\  e^{n+2}((p\bimp q)\to(p\land q))$ for $n\ge0$.
\end{lemma}

\begin{proof}
We first prove that $\alpha_0\equiv e^2((p\bimp q)\to(p\land q))$. By Lemma~\ref{lem:subst}, $\bg_n\equiv e^n(q\to p)$, $\cg_n\equiv e^n(p\to q)$, and $\bg_n\land\cg_n\ \equiv\ e^n(p\bimp q)$. Therefore,
\[
\ag_0\equiv e^2(p\bimp q)\to(e(p\to q)\lor e(q\to p)).
\]
Hence, it is sufficient to show
\[
e^2(p\bimp q)\to(e(p\to q)\lor e(q\to p))\ \equiv\ e^2((p\bimp q)\to(p\land q)).
\]
For this it is enough to show
\[
e(p\bimp q)\to((p\to q)\lor(q\to p))\ \equiv\ e((p\bimp q)\to(p\land q)).
\]
This is equivalent to showing
\begin{equation}\label{eq:ltor}
e(p\bimp q)\to((p\to q)\lor(q\to p))\ \vdash\ e(p\bimp q)\to e(p\land q)
\end{equation}
and
\begin{equation}\label{eq:rtol}
e(p\bimp q)\to e(p\land q)\ \vdash\ e(p\bimp q)\to((p\to q)\lor(q\to p)).
\end{equation}
But \eqref{eq:ltor} follows from
\begin{equation}\label{eq:ltorsimp}
e(p\bimp q)\land((p\to q)\lor(q\to p))\ \vdash\ e(p\land q)
\end{equation}
and \eqref{eq:rtol} follows from
\begin{equation}\label{eq:rtolsimp}
e(p\land q)\ \vdash\ (p\to q)\lor(q\to p).
\end{equation}
Both \eqref{eq:ltorsimp} and \eqref{eq:rtolsimp} in turn follow from Lemma~\ref{lem:epq}, proving that
 $\alpha_0\equiv e^2((p\bimp q)\to(p\land q))$. Thus, by Lemma~\ref{lem:monotbc}, $\alpha_n\equiv e^{n+2}((p\bimp q)\to(p\land q))$.
\end{proof}

\begin{lemma}\label{cor:xtoa}
If $e^{n+2}(p\bimp q)\vdash x$, then $x\to\ag_n\ \equiv\  \ag_n$.
\end{lemma}

\begin{proof}
By Lemma~\ref{cor:apq},
\begin{align*}
x\to\ag_n&\ \equiv\ x\to((e^{n+2}(p\bimp q)\to e^{n+2}(p\land q))\\
&\ \equiv\ (x\land e^{n+2}(p\bimp q))\to e^{n+2}(p\land q).
\end{align*}
By assumption, $x\land e^{n+2}(p\bimp q)\equiv e^{n+2}(p\bimp q)$. Therefore,
\[
x\to\ag_n\ \equiv\ e^{n+2}(p\bimp q)\to e^{n+2}(p\land q)\ \equiv\ \ag_n.
\]
\end{proof}

\begin{lemma}\label{cor:daimp}
$\ag_{m+n+1}\to\ag_n\ \equiv\ \dg_{m+n}\to\ag_n\ \equiv\ \ag_n$ for all  $m,n\ge0$.
\end{lemma}

\begin{proof}
We first show that $e^{n+2}(p\bimp q)\ \vdash\ \ag_{m+n+1}$ for all $m,n\ge0$. By Lemmas~\ref{lem:subst} and~\ref{cor:apq},
\[
\ag_{m+n+1}\ \equiv\ e^{m+n+1}(\alpha_0)\ \equiv\ e^{m+n+3}(p\bimp q)\to e^{m+n+3}(p\land q).
\]
Therefore, it suffices to show that $e^{n+2}(p\bimp q)\ \vdash\ e^{m+n+3}(p\land q)$. By Lemma~\ref{lem:epq},
\[
e^{m+n+3}(p\land q)\ \equiv\ e^{m+n+2}(p\lor q\lor (p\bimp q)).
\]
Thus, it suffices to show that $e^{n+2}(p\bimp q)\ \vdash\ e^{m+n+2}(p\bimp q)$. But this follows from Lemma~\ref{lem:ebimp}. Consequently,
$e^{n+2}(p\bimp q)\ \vdash\ \ag_{m+n+1}$ for all $m,n\ge0$. Applying
Lemma~\ref{cor:xtoa}
then yields $\ag_{m+n+1}\to\ag_n\ \equiv\ \ag_n$. This together with the definition of $\dg_{m+n}$ gives
\begin{eqnarray*}
\dg_{m+n}\to\ag_n\ & \equiv\ & (\ag_{m+n}\lor\ag_{m+n+1})\to\ag_n \\
\ & \equiv\ & (\ag_{m+n}\to\ag_n)\land(\ag_{m+n+1}\to\ag_n) \\
\ & \equiv\ & (\ag_{m+n}\to\ag_n)\land\ag_n \\
\ & \equiv\ & \ag_n,
\end{eqnarray*}
completing the proof.
\end{proof}

\begin{lemma}\label{lem:monot}
For all $m,n\ge0$ we have $(\ag_0\to\dg_1)\to\dg_0\ \vdash\ \dg_{m+n}\to\dg_m$.
\end{lemma}

\begin{proof}
We have $\dg_1\ \vdash\ \ag_0\to\dg_1$. Also, $\ag_0\to\dg_1\ \vdash\ \dg_0$ by assumption. Therefore, $\dg_1\vdash\dg_0$. By Lemma~\ref{lem:subst}, $\dg_m\equiv e^m(\dg_0)$. Thus, $\dg_{m+1}\vdash\dg_m$ for all $m\ge0$. Thus,
\[
\dg_{m+n}\ \vdash\ \dg_{m+n-1}\ \vdash\ \cdots\ \vdash\ \dg_{m+1}\ \vdash\ \dg_m.
\]
\end{proof}

\begin{proposition}\label{lem:aad}
From $(\alpha_0\to\delta_1)\to\delta_0$ and $\varkappa_0$ it follows that
\begin{equation}\label{eq:aad}
\ag_m\lor\ag_{m+n+1}\ \equiv\ \dg_m
\end{equation}
and
\begin{equation}\label{eq:add}
\ag_m\lor\dg_{m+n}\ \equiv\ \dg_m
\end{equation}
for all $m,n\ge0$.
\end{proposition}

\begin{proof}
It follows from Lemma~\ref{lem:monot} that $\dg_n\vdash\dg_0$ for all $n\ge0$, so $\ag_n\lor\ag_{n+1}\vdash\ag_0\lor\ag_1$, and hence $\ag_{n+1}\ \vdash\ \ag_0\lor\ag_1$. Therefore,
\begin{equation}\label{eq:kto1}
\ag_0\lor\ag_{n+1}\ \vdash\ \ag_0\lor\ag_1.
\end{equation}
On the other hand, from $\kg_0$ it follows that $\ag_1\vdash\ag_0\lor\bg_2$, and Lemma~\ref{lem:monotbc} implies that $\bg_2\vdash\bg_{n+2}$ and $\bg_{n+2}\vdash\ag_{n+1}$. Thus, $\ag_1\ \vdash\ \ag_0\lor\ag_{n+1}$. Consequently,
\begin{equation}\label{eq:1tok}
\ag_0\lor\ag_1\ \vdash\ \ag_0\lor\ag_{n+1}.
\end{equation}
From \eqref{eq:kto1} and \eqref{eq:1tok} it follows that $\ag_0\lor\ag_{n+1}\ \equiv\ \dg_0$ for all $n\ge0$. Applying $e^m$ yields \eqref{eq:aad}.

For \eqref{eq:add}, $\ag_m\lor\dg_{m+n}\ \equiv\ \ag_m\lor\ag_{m+n}\lor\ag_{m+n+1}$. By doubling $\ag_{m+n+1}$, applying \eqref{eq:aad}, and using the definition of $\dg_{m+n}$ we obtain:
\begin{eqnarray*}
\ag_m\lor\ag_{m+n}\lor\ag_{m+n+1} &\ \equiv &\ \ag_m\lor\ag_{m+n+1}\lor\ag_{m+n}\lor\ag_{m+n+1} \\
&\ \equiv &\ \dg_m\lor\ag_{m+n}\lor\ag_{m+n+1}\ \equiv\ \dg_m\lor\dg_{m+n},
\end{eqnarray*}
which in turn is equivalent to $\dg_m$ by Lemma~\ref{lem:monot}.
\end{proof}

\section{Incompleteness}

In this final section we prove our main incompleteness results. Our main technical tool is Lemma~\ref{lem:Sthend0} in which we show
that if a bi-Heyting algebra from $\mathcal V(\Sh)$ is complete, then it lies in a proper subvariety of $\mathcal V(\Sh)$.
From this we derive that no variety of bi-Heyting algebras in the interval $[\mathcal V(\F),\mathcal V(\Sh)]$ is generated
by its complete algebras. This yields continuum many BH-logics that are topologically incomplete, as well as continuum many
intermediate logics that are incomplete with respect to complete bi-Heyting algebras.

\begin{lemma}\label{lem:genbelow}
Let $X$ be an Esakia space and $A,B \subseteq X$. If $A\subseteq\dn B$, then $\Cl(A)\subseteq\dn\Cl(B)$.
\end{lemma}

\begin{proof}
We have $A\subseteq\dn B\subseteq\dn\Cl(B)$. The downset of a closed set is closed in every Esakia space (in fact, in every compact ordered space where the order relation is closed). Thus, $\Cl(A)\subseteq\Cl\dn\Cl(B)=\dn\Cl(B)$.
\end{proof}

\begin{lemma}\label{lem:Sthend0}
Let $\fB$ be a complete bi-Heyting algebra. If $\fB\models\Sh$, then $\fB\models\delta_0$.
\end{lemma}

\begin{proof}
Let $X$ be the dual BH-space of $\fB$. By Theorem~\ref{thm:cBH}, $X$ is extremally order-disconnected. We identify the elements of $\fB$ with the clopen upsets of $X$.

Suppose $\fB\models\Sh$ and $\fB\not\models\delta_0$. Then there is a valuation $\sV$ on clopen upsets of $X$ such that $\sV(\delta_0)\neq X$. We will utilize Theorem~\ref{thm:bb2kripke} to refute $\bb_2$ on $\fB$, thereby arriving at a contradiction.

Let
\[
a_i=\bigwedge_n\sV(\ag_{3n+i}),\quad i=0,1,2.
\]
Since in a complete bi-Heyting algebra finite joins distribute over infinite meets, we have
\begin{equation}\label{eq:bidistr}
a_0\lor a_1 = \left(\bigwedge_n\sV(\ag_{3n})\right) \vee \left(\bigwedge_m\sV(\ag_{3m+1})\right) = \bigwedge_{n,m}\left(\sV(\ag_{3n})\lor\sV(\ag_{3m+1})\right).
\end{equation}
Therefore, by Proposition \ref{lem:aad} (see \eqref{eq:aad}),
\[
a_0\lor a_1=\bigwedge_{n,m}\sV(\ag_{3n}\lor\ag_{3m+1})=\bigwedge_{n,m}\sV(\dg_{\min(3n,3m+1)}).
\]
By Lemma~\ref{lem:monot}, if $\fB\models\Sh$, then $\sV(\dg_n)\subseteq\sV(\dg_m)$ for $m<n$. It follows that the set $\{\sV(\dg_{\min(3n,3m+1)})\mid m,n\ge0\}$ contains infinitely many elements from the decreasing sequence $\{\sV(\dg_k)\mid k\ge0\}$. Thus, $\bigwedge_{n,m}\sV(\dg_{\min(3n,3m+1)})=\bigwedge_k\sV(\dg_k)$, and so $a_0\lor a_1 = \bigwedge_k\sV(\dg_k)$.
Similarly, denoting $\bigwedge_k\sV(\dg_k)$ by $d$, we have $a_0\lor a_2=a_1\lor a_2=d$.


Let $\cV{\ag_n}=X\setminus\sV(\ag_n)$, $\cV{\dg_n}=X\setminus\sV(\dg_n)$, $A_i=X\setminus a_i$, and $D=X\setminus d$ be the corresponding clopen downsets. Then
\[
A_0\cap A_1=A_0\cap A_2=A_1\cap A_2=D.
\]
Moreover, $D\neq\emp$ since $\sV(\delta_0)\neq X$. By Theorem~\ref{thm:bb2kripke}, to show that $\bb_2$ is refuted on $\fB$, it only remains to show that $D$ is nowhere cofinal in $A_0, A_1, A_2$.


We show that $D$ is nowhere cofinal in $A_0$. 
Since we are in an extremally order-disconnected BH-space, it follows from (\ref{eq:bihjoinmeet}) that
\begin{equation}\label{eq:meetint}
\bigwedge_n\sV(\dg_n)=\Int\left(\bigcap_n\sV(\dg_n)\right).
\end{equation}
Therefore,
\[
D=X\setminus\bigwedge_n\sV(\dg_n)=
X\setminus\Int\left(\bigcap_n\sV(\dg_n)\right)
=\Cl\left(\bigcup_n\cV{\dg_n}\right).
\]
Similarly,
\[
A_0=X\setminus\bigwedge_n\sV(\ag_{3n})=\Cl\left(\bigcup_n\cV{\ag_{3n}}\right).
\]

Let $D'=\bigcup_n\cV{\dg_n}$, $a'_n=\cV{\ag_n}\setminus\cV{\dg_n}$, and $A_0'=\bigcup_n a'_{3n}$.
We then have $D=\mathsf{cl}D'$ and $A_0' \subseteq A_0$.

\begin{claim}\label{claim}
$A_0'\cap D'=\emp$ and $D'\subseteq\dn A'_0$.
\end{claim}
\begin{proof}[Proof of claim]
We first show that $A_0'\cap D'=\emp$. \color{black} We have
\[
a'_{3n}\cap\cV{\dg_m}= (\cV{\ag_{3n}}\setminus\cV{\dg_{3n}})\cap\cV{\dg_m} =
(\cV{\ag_{3n}}\cap\cV{\dg_m})\setminus\cV{\dg_{3n}}.
\]
We show that $a'_{3n}\cap\cV{\dg_m} = \varnothing$.
Since $\fB\models\Sh$, we can apply  Lemma~\ref{lem:monot} and Proposition~\ref{lem:aad}. First suppose that $m\le3n$. Then Lemma~\ref{lem:monot} yields that $\cV{\dg_m}\setminus\cV{\dg_{3n}}=\emp$. Therefore, $a'_{3n}\cap\cV{\dg_m} = \varnothing$. Next suppose that $m>3n$. By Proposition~\ref{lem:aad} (see \eqref{eq:add}),
\[
\cV{\ag_{3n}}\cap\cV{\dg_m}=\cV{\ag_{3n}\lor\dg_m}=\cV{\dg_{3n}}.
\]
Thus,
$a'_{3n}\cap\cV{\dg_m}=\emp$ in this case too. This implies that $A_0'\cap D'=\emp$.

We next show that $D'\subseteq\dn A'_0$. If $x\in D'=\bigcup_n\cV{\dg_n}$, then $x\in\cV{\dg_{3k}}$ for some $k$ (by Lemma~\ref{lem:monot}). Hence,  $x\in\cV{\ag_{3k}}$ as $\dg_{3k}=\ag_{3k}\lor\ag_{3k+1}$, so $\cV{\dg_{3k}}\subseteq\cV{\ag_{3k}}$. By Lemma~\ref{cor:daimp}, $\dg_{3k}\to\ag_{3k}\equiv \ag_{3k}$. Therefore,
\[
\dn(\cV{\ag_{3k}}\setminus\cV{\dg_{3k}})=\cV{\dg_{3k}\to\ag_{3k}}=\cV{\ag_{3k}},
\]
and so $x\in\dn(\cV{\ag_{3k}}\setminus\cV{\dg_{3k}})\subseteq\dn A'_0$.
\end{proof}


We are ready to prove that $D$ is nowhere cofinal in $A_0$. By Remark~\ref{rem: nowhere dense}(1), it is sufficient to find $S\subseteq A_0$ with $D\cap S=\emp$ and $D\subseteq\dn S$. Let $S=\closure(A'_0)$. As $A_0'\subseteq A_0$ and $A_0$ is clopen, $\closure(A_0')\subseteq A_0$. Therefore, $S \subseteq A_0$. By Claim~\ref{claim}, $A'_0\cap D'=\emp$. Since $A'_0$ is open, we have $A'_0\cap\mathsf{cl}(D')=\emp$, so $A'_0\cap D=\emp$. Because $D$ is clopen, $\mathsf{cl}(A'_0)\cap D=\emp$. Thus, $S\cap D=\emp$. Finally, $D'\subseteq{\downarrow}A_0'$ by Claim~\ref{claim}. Therefore, we may apply Lemma \ref{lem:genbelow} to obtain that $D=\closure(D')\subseteq{\downarrow}\closure(A_0')=\dn S$. Thus, $D$ is nowhere cofinal in $A_0$.
The proof that $D$ is nowhere cofinal in $A_1$ and $A_2$ is similar.

Consequently, $(**)$ does not hold in $X$. By Theorem~\ref{thm:bb2kripke}, this means that $\bb_2$ is refuted on $\fB$, contradicting $\fB\models\Sh$.
\end{proof}


\begin{theorem}\label{thm:incompleteness}
No variety in the interval $[\mathcal V(\F),\mathcal V(\Sh)]$ is generated by its complete algebras.
\end{theorem}

\begin{proof}
Let $\mathcal V\in[\mathcal V(\F),\mathcal V(\Sh)]$ and let $\fB$ be a complete algebra in $\mathcal V$. Since $\mathcal V\subseteq\mathcal V(\Sh)$, it follows from Lemma~\ref{lem:Sthend0} that $\fB\models\delta_0$. On the other hand, since $\mathcal V(\F)\subseteq\mathcal V$, it follows from Corollary~\ref{cor:ASnod0} that $\mathcal V\not\models\delta_0$. Thus, $\mathcal V$ cannot be generated by its complete algebras.
\end{proof}

\begin{corollary}\label{cor:top-incomplete}
No logic in the interval $[\Sh,\F]$ is topologically complete.
\end{corollary}

\begin{proof}
Let ${\sf L}\in[\Sh,\F]$ and let $X$ be a topological space such that the algebra $\Op(X)$ of open subsets of $X$ is a bi-Heyting algebra validating $\sf L$. It follows from the proof of Theorem~\ref{thm:incompleteness} that ${\sf L}\not\vdash\delta_0$ but that $\Op(X)\models\delta_0$. Thus, $\sf L$ is topologically incomplete.
\end{proof}

In \cite{Lit02} Litak utilized the Jankov-Fine technique of frame formulas to prove that the interval $[\She,\Fi]$ contains continuum many intermediate logics. Since this technique is also applicable to HB-logics,
as an immediate consequence of Litak's result and Corollary~\ref{cor:top-incomplete} we obtain:

\begin{corollary}
There are continuum many extensions of $\HB$ that are topologically incomplete.
\end{corollary}


Another consequence of our results is the following:

\begin{corollary}\label{cor:nobih}
No variety of Heyting algebras in $[{\mathcal V}(\Fi),{\mathcal V}(\She)]$ is generated by its complete bi-Heyting algebras.
\end{corollary}

\begin{proof}
Let $\mathcal V\in[\mathcal V(\Fi),\mathcal V(\She)]$ and let $\fB \in \mathcal V$ be a complete bi-Heyting algebra. Then $\fB\models\delta_0$ by Lemma~\ref{lem:Sthend0}, and $\mathcal V\not\models\delta_0$ by Corollary~\ref{cor:ASnod0}. Thus, $\mathcal V$ cannot be generated by its complete bi-Heyting algebras.
\end{proof}

As an immediate consequence of Corollary \ref{cor:nobih} we obtain:

\begin{corollary}\label{cor:Massas}
No intermediate logic in the interval $[\She,\Fi]$ is complete with respect to complete bi-Heyting algebras.
\end{corollary}

\begin{remark}
Guillaume Massas notified us that he also proved a result analogous to Corollary~\ref{cor:Massas} using different technique.
\end{remark}

We conclude the paper by pointing out that our technique does not generalize directly to the setting of complete Heyting algebras. Indeed, there are two key steps in the proof of Lemma~\ref{lem:Sthend0} and both require that the complete algebra $\fB$ is bi-Heyting. More specifically, in \eqref{eq:bidistr} we use that finite joins distribute over infinite meets; this is not true in a complete Heyting algebra that is not bi-Heyting. In \eqref{eq:meetint} we use that the interior of a closed upset is a clopen upset; this is not true in an extremally order-disconnected Esakia space that is not a BH-space. Thus, further insight is needed to attack Kuznetsov's original problem.


\bibliographystyle{amsplain}
\providecommand{\bysame}{\leavevmode\hbox to3em{\hrulefill}\thinspace}
\providecommand{\MR}{\relax\ifhmode\unskip\space\fi MR }
\providecommand{\MRhref}[2]{%
  \href{http://www.ams.org/mathscinet-getitem?mr=#1}{#2}
}
\providecommand{\href}[2]{#2}

\end{document}